\documentclass[a4paper,12pt]{article}
\usepackage{amsfonts}
\usepackage{amssymb}
\usepackage{latexsym}
\usepackage{amsmath}
\usepackage{amsthm}
\usepackage{graphicx}

\newtheorem{theorem}{Theorem}

\newtheorem{proposition}[theorem]{Proposition}

\newtheorem{remark}{Remark}

\newcommand{\NN}{\mathbb{N}}
\newcommand{\RR}{\mathbb{R}}
\newcommand{\ZZ}{\mathbb{Z}}

\newcommand{\bse}{\boldsymbol{e}}
\newcommand{\bsx}{\boldsymbol{x}}
\newcommand{\bsg}{\boldsymbol{g}}
\newcommand{\bsh}{\boldsymbol{h}}
\newcommand{\bsb}{\boldsymbol{b}}

\newcommand{\bszero}{\boldsymbol{0}}

\newcommand{\cP}{\mathcal{P}}

\title{A note on isotropic discrepancy and spectral test of lattice point sets}
\author{Friedrich Pillichshammer and Mathias Sonnleitner\thanks{The authors are supported by the Austrian Science Fund (FWF): Projects F5509-N26 (Pillichshammer) and F5513-N26 (Sonnleitner), which are part of the Special Research Program ``Quasi-Monte Carlo Methods: Theory and Applications''.}}

\date{}

\begin{document}

\maketitle

\begin{abstract}

We show that the isotropic discrepancy of a lattice point set can be bounded from below and from above in terms of the spectral test of the corresponding integration lattice.  From this result we deduce that the isotropic discrepancy of any $N$-element lattice point set in $[0,1)^d$ is at least of order $N^{-1/d}$. This order of magnitude is best possible for lattice point sets in dimension $d$.
\end{abstract}

\centerline{\begin{minipage}[hc]{130mm}{
}
\end{minipage}}

\section{Introduction}

For a finite set of points $\cP=\{\bsx_0,\bsx_1,\ldots,\bsx_{N-1}\}$ in the $d$-dimensional unit-cube $[0,1)^d$ the {\it isotropic discrepancy} is defined as $$J_N(\cP):=\sup_{C}\left|\frac{\#\{n \ : \ 0 \le n < N,\ \bsx_n \in C\}}{N}-{\rm volume}(C)\right|,$$ where the supremum is extended over all convex subsets $C$ of $[0,1)^d$. The isotropic discrepancy is a quantitative measure for the irregularity of distribution of $\cP$, see, e.g., \cite{DT97,kuinie}.

In \cite[Theorem~1]{sch75} (see also \cite[Theorem~13A]{sch77}) Schmidt proved the following general lower bound for arbitrary point sets: {\it For every dimension $d$ there exists a positive constant $c_d$ such that for all $N$-element point sets $\cP$ in $[0,1)^d$ we have $$J_N(\cP) \ge \frac{c_d}{N^{2/(d+1)}}.$$} This result is essentially (up to $\log$-factors) best possible as shown by Beck~\cite{beck} for $d=2$ and Stute~\cite{stute} for $d\ge 3$ using probabilistic methods.

Upper bounds on the isotropic discrepancy of special point sets are given in \cite{lar86,lar88}. For example, $J_N(\cP) = O(N^{-1/d})$ if $\cP$ is the Hammersley net or the initial segment of the Halton sequence in dimension $d$. Aistleitner, Brauchart and Dick \cite{ABD} used plane point sets with low isotropic discrepancy to generate point sets on the sphere $\mathbb{S}^2$ with small spherical cap discrepancy. They analyzed $(0,m,2)$-nets in base $b$ in the sense of Niederreiter \cite{niesiam} and Fibonacci lattices and showed for these classes of point sets $(d=2)$ an isotropic discrepancy of order of magnitude $O(N^{-1/2})$ whereas the optimal rate would be $N^{-2/3}$. Furthermore, they stated the following question: {\it Whether $(0,m,2)$-nets and/or Fibonacci lattices achieve the optimal rate of convergence
for the isotropic discrepancy is an open question} (see \cite[p.~1001]{ABD}).  

In this note we present a very short and elementary argument that shows that the answer to this question is negative for the Fibonacci lattice. Even more generally, we show that the same applies to any lattice point set in arbitrary dimension $d$ arising from the intersection of an integration lattice with the unit cube $[0,1)^d$. We acknowledge that for the special case of rank-1 lattice point sets (in particular for the Fibonacci lattice) this can already be deduced from a result by Larcher~\cite{lar89} on initial segments of Kronecker sequences, which, however, is proved with different methods and shows that there are large empty rotated boxes. 

Our proof is based on the observation that the isotropic discrepancy of a lattice point set is up to constants depending on the dimension $d$ exactly the spectral test of the corresponding integration lattice.
In the following section we present basic information on integration lattices and lattice point sets. Our results together with their proofs can be found in Section~\ref{sec3}.

\section{Integration lattices}

A $d$-dimensional {\it lattice} is obtained by taking a basis $\bsb_1,\ldots,\bsb_d$ of the vector space $\RR^d$ and forming the set
$$
L=\left\{\sum_{i=1}^d k_i \bsb_i : k_i \in \ZZ \;  \mbox{ for } 1 \le i \le d \right\}
$$
of all $\ZZ$-linear combinations of $\bsb_1,\ldots,\bsb_d$ with integer coefficients. A $d$-dimensional {\it integration lattice} is a lattice that contains $\ZZ^d$ as a subset. For more information we refer to the books \cite{niesiam,SJ}.

If $L$ is a $d$-dimensional integration lattice, then the intersection $L \cap [0,1)^d$ is a finite set since $L$ is discrete, and this finite set of points in $[0,1)^d$ forms a so called {\it lattice point set} denoted by $\cP(L)$. Lattice point sets are commonly used as nodes for numerical integration rules (so called lattice rules), see, e.g., \cite{DKS,niesiam,SJ}. An important sub-class of lattice point sets are so-called {\it rank-1 lattice point sets} which consist of the elements $\{(n/N) \bsg\}$ for $n=0,1,\ldots,N-1$, where $N \in \NN$ and $\bsg$ is a suitable lattice point in $\ZZ^d$. Here, $\{\cdot\}$ denotes the fractional-part-function applied to each component of a vector in $\RR^d$. For example, the aforementioned Fibonacci lattice ($d=2$) is obtained when choosing $N=F_m$, the $m^{{\rm th}}$ Fibonacci number, and letting $\bsg=(1,F_{m-1})$.

An important concept for the analysis of lattice point sets is that of the dual lattice: The {\it dual lattice} $L^{\bot}$ of a $d$-dimensional integration lattice $L$ is defined by
$$
L^{\bot}=\left\{\bsh \in \RR^d\ :\ \bsh \cdot \bsg \in \ZZ \; \mbox{ for all } \bsg \in L\right\},
$$
where $\cdot$ denotes the standard inner product on $\RR^d$.

A well known numerical quantity to assess the coarseness of (integration) lattices $L$ in $\RR^d$ is the {\it spectral test} which is defined as $$\sigma(L):=\frac{1}{\min\{\|\bsh\|_2 \, : \, \bsh \in L^{\bot}\setminus\{\bszero\}\}},$$ where $\|\cdot\|_2$ denotes the $\ell_2$-norm in $\RR^d$. It has the following geometric interpretation: The spectral test is the maximal distance between two adjacent hyperplanes, taken over all families of parallel hyperplanes that cover the lattice $L$, see \cite{He98} (and also \cite[p.~29]{SJ}).

\section{The results and the proofs}\label{sec3}

We will show the following result:

\begin{theorem}\label{thm}
	Let $\cP(L)$ be an $N$-element lattice point set in $[0,1)^d$. Then we have $$J_N( \cP(L)) \ge \min\Big(\frac{1}{2\sqrt{d}+1},\frac{c_d}{N^{1/d}}\Big),$$ where $$c_d:=  \frac{\sqrt{\pi}}{2\sqrt{d}+1} \left(\Gamma\left(\frac{d}{2}+1\right)\right)^{-1/d},$$ and where $\Gamma$ denotes the Gamma function.  If $\sigma(L)\leq 1/2$, then $c_d$ may be replaced by $c \left(\Gamma\left(\frac{d}{2}+1\right)\right)^{-1/d}$ for some $c>0$ that is independent of $d$.
\end{theorem}

We remark that it is easily seen that $$c_d \sim \sqrt{\frac{\pi\, {\rm e}}{2}} \, \frac{1}{d} \ \ \ \mbox{ as $d \rightarrow \infty$.}$$ 

\begin{remark}
	In the original version of this preprint (and its published version \cite{PS20}) there appeared a slight inaccuracy in the presented lower bound in Theorem~\ref{thm} and also in the following Theorem~\ref{pr1}. Also the upper bound in Theorem~\ref{pr1} is actually larger by roughly a factor of $2^d$, which is corrected in the follow-up paper \cite{SP21}.
\end{remark}

In order to prove Theorem~\ref{thm} we first show that the isotropic discrepancy of lattice point sets can be lower and upper bounded by means of the spectral test of the corresponding integration lattices. This result is interesting in its own right. Then we present a lower bound on the spectral test of integration lattices.

\begin{theorem}\label{pr1}
	Let $\cP(L)$ be an $N$-element lattice point set in $[0,1)^d$. Then we have $$\frac{\sigma(L)}{\sqrt{d}+\sigma(L)}\leq J_N(\mathcal{P}(L))\leq d^2 \, 2^{d} \sigma(L).$$ If $\sigma(L)\leq 1/2$, then the lower bound can be replaced by $c\, \sigma(L)$, where $c>0$ is an absolute constant.
\end{theorem}

\begin{proof}
We show the lower bound on the isotropic discrepancy by finding a convex set of zero volume on which at least $N\sigma(L)/(\sqrt{d}+\sigma(L))$ points lie. Intersecting the lattice point set with a suitable hyperplane gives a convex set as desired.

Finding such an hyperplane is basically an application of the pigeonhole principle. Namely, by the geometric interpretation of the spectral test, there exists a family of parallel hyperplanes covering the associated lattice $L$ and any two hyperplanes in the family are separated by $\sigma(L)$. Since the unit cube $[0,1)^d$ has diameter $\sqrt{d}$ and $\bszero\in\cP(L)$, it can be intersected by no more than $\sqrt{d}/\sigma(L)+1$ such hyperplanes. 

If all hyperplanes intersecting $[0,1)^d$ contained strictly less than $N\sigma(L)/(\sqrt{d}+\sigma(L))$ points of the lattice point set, this would be a contradiction since any lattice point of $\cP(L)$ must lie on one of these hyperplanes. Therefore, we may find at least $N\sigma(L)/(\sqrt{d}+\sigma(L))$ lattice points lying on some hyperplane. This implies the lower bound.\\

Suppose now that additionally $\sigma(L)\leq 1/2$. We show the lower bound on the isotropic discrepancy by finding an empty convex set $C\subset [0,1)^d$ of volume at least $c\sigma(L)$ which will be constructed by intersecting the slab between two suitable hyperplanes with the unit cube.

To bound its volume from below, we make use of recent progress due to K\"{o}nig and Rudelson \cite{KR19}. Theorem 1.1. in their paper implies the existence of a $c>0$ such that for any $d\in\NN$ the $(d-1)$-dimensional volume of the intersection of the cube $[0,1)^d$ with any hyperplane having distance at most $1/2$ from its center is bounded from below by $c$. 

To apply this theorem, fix a covering of $\cP(L)$ by a family $\mathcal{H}$ consisting of parallel hyperplanes which are separated by at least $\sigma(L)$. We find a pair of hyperplanes surrounding the center of the cube by a pigeonhole argument. 

To this end, consider the one-dimensional space orthogonal to all hyperplanes in $\mathcal{H}$ which is spanned by some $\bsh\in\RR^d$. The rays emanating from the center of the cube into the directions $\pm \bsh$ hit an hyperplane of $\mathcal{H}$ at distance at most $\sigma(L)$ from the center of the cube. In this way, we get a pair of adjacent hyperplanes $H_1,H_2\in\mathcal{H}$ sandwiching the center of the cube and having distance $\sigma(L)$. Denote the collection of all hyperplanes parallel to the family $\mathcal{H}$ which lie between $H_1$ and $H_2$ by $\widetilde{\mathcal{H}}$. Finally, define the open convex set $$C:=\text{int}(\text{conv}(H_1\cup H_2)\cap [0,1)^d),$$ which does not contain any point from $\cP(L)$. The volume of $C$ can be bounded from below by $$\text{volume}(C)\geq \sigma(L) \inf_{H\in \widetilde{\mathcal{H}}} \text{volume}_{d-1}(H\cap [0,1)^d).$$ 

Since all hyperplanes in $\widetilde{\mathcal{H}}$ have distance at most $\sigma(L)\leq 1/2$ from the center, the lower bound of K\"{o}nig and Rudelson yields that the infimum is bounded from below by some $c>0$. Thus, we have found an empty convex set $C$ with $\text{volume}(C)\geq c\,\sigma(L)$ and the first inequality is proven. \\

Now we prove the upper bound. By means of the LLL-algorithm, see, e.g., \cite[Chapter 17]{galbraith}, we find a reduced basis $(\bsb_1,...,\bsb_d)$ of the lattice $L$ containing short near-orthogonal vectors. The definition of a reduced basis requires the Gram-Schmidt orthogonalization $(\bsb^{\ast}_1,...,\bsb^{\ast}_d)$ which is obtained from the lattice basis by setting $\bsb^{\ast}_1=\bsb_1$ and
\[
\bsb^{\ast}_i=\bsb_i-\sum_{j=1}^{i-1} \mu_{i,j} \bsb^{\ast}_j \quad \text{for } 2\leq i \leq d,
\]
where 
\[
\mu_{i,j}=\frac{\bsb_i\cdot\bsb^{\ast}_j}{\Vert \bsb^{\ast}_j\Vert_2^2}\quad \text{for } 2\leq i \leq d \text{ and } 1\leq j \leq i-1.
\]

Then, from the properties of a reduced basis it can easily be deduced, see, e.g., \cite[Lemma 17.2.8]{galbraith}, that
\begin{enumerate}
\item[(a)] $\Vert \bsb^{\ast}_j\Vert_2^2\leq 2^{i-j}\Vert \bsb^{\ast}_{i}\Vert_2^2$ for $1\leq j \leq i \leq d$ and
\item[(b)] $\Vert \bsb_i\Vert_2^2\leq 2^{d-1}\Vert \bsb^{\ast}_i\Vert_2^2$ for $1\leq i\leq d$.
\end{enumerate}

Together these estimates imply
\[
\Vert \bsb^{\ast}_d\Vert_2\geq 2^{-(d-1)/2}\max_{1\leq i\leq d}\Vert \bsb^{\ast}_i\Vert_2\geq 2^{-d+1}\max_{1\leq i \leq d}\Vert \bsb_i\Vert_2.
\]
Now consider the fundamental parallelotope associated to the basis $(\bsb_1,...,\bsb_d)$, i.e.,
\begin{equation}\label{fundpara}
P:=\left\{\sum_{i=1}^d \lambda_i \bsb_i \ : \ 0\leq \lambda_i <1 \text{ for } 1\leq i\leq d\right\},
\end{equation}
which is often called a {\it unit cell} of the lattice $L$. This unit cell induces a partition of $\RR^d$ into disjoint cells $\bsx+P$ where $\bsx\in L$. Each of these translated unit cells contains only the lattice point $\bsx$ and has diameter
\[
\text{diam}(P)\leq \sum_{i=1}^d \Vert \bsb_i\Vert_2 \leq d \max_{1\leq i \leq d} \Vert \bsb_i\Vert_2 \leq d \, 2^{d-1} \, \Vert \bsb_d^{\ast}\Vert_2.
\]
By construction we have 
\[
\Vert \bsb_d^{\ast}\Vert_2 = \left\Vert \bsb_d - \sum_{i=1}^{d-1} \mu_{d,j}\bsb^{\ast}_j\right\Vert_2.
\]
That is, the length of the last vector in the Gram-Schmidt orthogonalization is the length of the projection of the vector $\bsb_d$ onto the orthogonal complement of the subspace span$\{\bsb^{\ast}_1,...,\bsb^{\ast}_{d-1}\}=\text{span}\{\bsb_1,...,\bsb_{d-1}\}$ spanned by the other basis vectors. But this is exactly the distance between two adjacent hyperplanes of the family of parallel hyperplanes
\[
k\bsb_d+\text{span}\{\bsb_1,...,\bsb_{d-1}\}\ \ \ \text{for } k\in\mathbb{Z},
\]
which covers the entire lattice $L$. Therefore, since this distance cannot be larger than the spectral test, we have
\[
d\, 2^{d-1}\, \sigma(L)\geq d\, 2^{d-1} \, \Vert \bsb_d^{\ast}\Vert_2 \geq \text{diam}(P).
\]

In the spirit of Aistleitner et al. (see \cite[Proof of Lemma~17]{ABD}) we now bound the discrepancy 
$$\Delta_{\cP(L)}(C):=\frac{\#\{n \ : \ 0 \le n < N,\ \bsx_n \in C\}}{N}-{\rm volume}(C)$$
of an arbitrary convex set $C\subseteq [0,1)^d$ in terms of the diameter of the unit cell $P$ given in \eqref{fundpara}. We describe the idea behind this. 

Consider the collection of translated unit cells $\bsx+P$ with $\bsx \in L$ that are fully contained in $C$ and denote their union by $W^{\circ}$. Likewise denote the union of translated unit cells having non-empty intersection with $C$ by $\overline{W}$. Clearly, we have $$W^{\circ}\subseteq C \subseteq \overline{W}.$$ 

Since $L$ is an integration lattice, the volume of $P$ is exactly $1/N$ (see, e.g., \cite{niesiam, SK}). Furthermore, every translated unit cell $\bsx+P$ with $\bsx \in L$ contains only the lattice point $\bsx$. This implies (see \cite{ABD}) that the discrepancy of $C$ is only influenced by cells intersecting the boundary of $C$ and satisfies
\[
|\Delta_{\cP(L)}(C)|\leq \max\{\text{volume}(\overline{W}\backslash C),\text{volume}(C\backslash W^{\circ})\}.
\]
The volume of these two differences can be bounded by the diameter of a unit cell times the surface area of the unit cube $[0,1)^d$. We already derived a bound for the first quantity and the second equals $2d$ which is the number of $(d-1)$-dimensional faces of $[0,1)^d$. 

Summarizing, we arrive at
\[
|\Delta_{\cP(L)}(C)|\leq d^2\, 2^{d} \sigma(L).
\]
Since the convex set $C$ was arbitrary, this also holds for the supremum.
\end{proof}

\begin{remark}\rm
The proof of the lower bound did not require that $L$ is an integration lattice and therefore works for any $N$-point set arising from intersecting a $d$-dimensional lattice with $[0,1)^d$. 
\end{remark}

The lower bound in Theorem~\ref{pr1} together with the following proposition shows that the linear structure of lattice point sets, although of advantage for analytical and practical reasons, forces a large isotropic discrepancy.

\begin{proposition}\label{pr2}
Let $\cP(L)$ be an $N$-element lattice point set in $[0,1)^d$. Then we have $$\sigma(L) \ge \frac{\sqrt{\pi}}{2} \left(\Gamma\left(\frac{d}{2}+1\right)\right)^{-1/d} \frac{1}{N^{1/d}}.$$ 
\end{proposition}

The proof of Proposition~\ref{pr2} uses a standard argument in the theory of lattices that is based on Minkowski's fundamental theorem, which states the following: {\it 
Let $L$ be a lattice in $\RR^d$. Then any convex set in
$\RR^d$ which is symmetric with respect to the origin and with
volume greater than $2^d \det(L)$ contains a non-zero lattice point of $L$} (see, e.g.,  \cite[Theorem~447]{HardWr}). Here, $\det(L)$ is the so-called {\it determinant} of a lattice $L$ which is geometrically interpreted the volume of a unit cell (see, e.g.,  \cite{niesiam,SK}).

For the sake of completeness we give the short proof of Proposition~\ref{pr2}. 

\begin{proof}[Proof of Proposition~\ref{pr2}]
Let $L$ be the integration lattice yielding the $N$-element lattice point set and let $L^{\bot}$ be the corresponding dual lattice. According to \cite[Sec.~3 and 4]{SK} (see also \cite[Theorem~5.30]{niesiam}) we have $\det(L^{\bot})=N$. Now consider the centered $\ell_2$-ball $$C_r^d:=\{\bsx \in \RR^d \, : \, x_1^2+\cdots+x_d^2 \le r^2\}$$ of radius $r>0$. Then $C_r^d$ is symmetric
with respect to the origin and the volume of
$C_r^d$ is $${\rm Vol}(C_r^d)=r^d \frac{\pi^{d/2}}{\Gamma(\frac{d}{2}+1)}.$$ Hence, by
Minkowski's theorem applied to $L^\bot$, we have that if
$$r^d \frac{\pi^{d/2}}{\Gamma(\frac{d}{2}+1)} \ge 2^d \det(L^\bot) =2^d N,$$ i.e., if
$r \ge \frac{2}{\sqrt{\pi}} (\Gamma(\frac{d}{2}+1))^{1/d}  N^{1/d}=:\widetilde{r}(d,N)$, then $C_r^d$ contains a non-zero point from $L^\bot$. In other words, $L^\bot$ contains a non-zero lattice point which
belongs to $C_{\widetilde{r}(d,N)}^d$ and therefore we have
$(\sigma(L))^{-1} \le \widetilde{r}(d,N)$. Taking the reciprocal values, we obtain the desired result. 
\end{proof}
 
Combining the lower bound from Theorem~\ref{pr1} with Proposition~\ref{pr2}, we obtain the result of Theorem~\ref{thm}.\\

\begin{proof}[Proof of Theorem~\ref{thm}]
	If $\sigma(L)>1/2$ we obtain from Theorem~\ref{pr1} together with the fact that $x\mapsto x/(\sqrt{d}+x)$ is increasing for $x\ge 1/2$ that 
	\[
J_N(\mathcal{P}(L))
\ge \frac{\sigma(L)}{\sqrt{d}+\sigma(L)}
\ge \frac{1}{2\sqrt{d}+1}.
	\]
	If $\sigma(L)\le 1/2$, then it suffices to combine Theorem~\ref{pr1} with Proposition~\ref{pr2}. 
\end{proof}

It can be shown that the order of magnitude $N^{-1/d}$ is best possible for the spectral test (or for the isotropic discrepancy) of integration lattices in dimension $d$. This bound can be even achieved for rank-1 lattice point sets.

\begin{proposition}
For every dimension $d$ there exists a positive number $C_d$ depending only on $d$ with the following property: for every prime number $N$ there exists a lattice point $\bsg \in \{0,1,\ldots,N-1\}^d$ such that the integration lattice corresponding to the rank-1 lattice point set $\cP(\bsg,N)=\{\{(n/N)\bsg\} \ : \ n=0,\ldots,N-1\}$ has a spectral test of at most $C_d N^{-1/d}$.  
\end{proposition}

\begin{proof}
The result follows from \cite[Proof of Lemma~2]{DLPW} together with the fact that the $\ell_2$- and the $\ell_1$-norm in $\RR^d$ are equivalent. 
\end{proof}

Also, an appropriately scaled version of $\ZZ^d$ is asymptotically optimal in terms of the spectral test and probably also the isotropic discrepancy. To see this, take the sequence $(M^d)_{M\in\NN}$ and consider the scaled integer lattice $(1/M)\ZZ^d$. It has $N=M^d$ points inside the unit cube $[0,1)^d$ and its dual lattice is $M\ZZ^d$. Therefore, the length of the shortest non-zero vector in its dual lattice is $M=N^{1/d}$ and the spectral test of the corresponding lattice point set is equal to $\sigma(\cP((1/M)\ZZ^d))=M^{-1}=N^{-1/d}$.\\

Just as easily one can construct sequences of bad lattice point sets in $[0,1)^d$ which concentrate on few hyperplanes. For example, let $N=2M^{d-1}$ for some $M\in \NN$. Take a lattice $L$ with generating vectors $\bsb_1,...,\bsb_d$, where $\bsb_i=(1/M)\bse_i$ for $1\leq i \leq d-1$ and $\bsb_d=(1/2) \bse_d$. Here, $\bse_i$ is the $i^{\rm{th}}$ unit vector. For any $M\in\NN$, the corresponding lattice point set contains $N$ points which lie on two hyperplanes of distance $\sigma(L)=1/2$ that are parallel to the span of $\bse_1,\ldots,\bse_{d-1}$. Hence, the box $(0,1)^{d-1} \times (0,1/2)$ is always empty and thus the isotropic discrepancy is at least $1/2$.

\paragraph{Acknowledgement.} The authors thank Mario Ullrich for valuable discussions.

\begin{small}
\noindent\textbf{Authors' addresses:}\\
\noindent  Friedrich Pillichshammer, Institut f\"{u}r Finanzmathematik und Angewandte Zahlentheorie, Johannes Kepler Universit\"{a}t Linz, Altenbergerstr.~69, 4040 Linz, Austria.\\
\textbf{E-mail:} \texttt{friedrich.pillichshammer@jku.at}\\

\noindent Mathias Sonnleitner, Institut f\"ur Analysis, Johannes Kepler Universit\"{a}t Linz, Altenbergerstr.~69, 4040 Linz, Austria.\\
\textbf{E-mail:} \texttt{math.s@posteo.net}\\
\end{small}

\end{document}